\documentclass{amsart}

\usepackage{lmodern} 
\usepackage{microtype} 
\usepackage[english]{babel} 
\usepackage{mathrsfs}

\usepackage{tikz} 
\usepackage{pgfplots}
\pgfplotsset{compat=1.17}
\usepgfplotslibrary{fillbetween}

\theoremstyle{plain} 
\newtheorem{theorem}{Theorem} 
\newtheorem{lemma}[theorem]{Lemma} 
\newtheorem{corollary}[theorem]{Corollary} 

\begin{document} 
\title{The best constant in a Hilbert-type inequality} 
\date{\today} 

\author{Ole Fredrik Brevig} 
\address{Department of Mathematics, University of Oslo, 0851 Oslo, Norway} 
\email{obrevig@math.uio.no}

\begin{abstract}
	We establish that
	\[\sum_{m=1}^\infty \sum_{n=1}^\infty a_m \overline{a_n} \frac{mn}{(\max(m,n))^3} \leq \frac{4}{3}\sum_{m=1}^\infty |a_m|^2\]
	holds for every square-summable sequence of complex numbers $a = (a_1,a_2,\ldots)$ and that the constant $4/3$ cannot be replaced by any smaller number. Our proof is rooted in a seminal 1911 paper concerning bilinear forms due to Schur, and we include for expositional reasons an elaboration on his approach.
\end{abstract}


\maketitle

\section{Introduction}
Set $\mathbb{R}_+=(0,\infty)$. Suppose that $K$ is a function defined on $\mathbb{R}_+ \times\, \mathbb{R}_+$ that is continuous, positive, symmetric, and homogeneous of degree $-1$, i.e. such that 
\[K(\lambda x, \lambda y) = \lambda^{-1} K(x,y)\]
holds for all $x,y,\lambda>0$. The problem of interest is to identify the smallest real number $C=C(K)$ such that the inequality
\begin{equation} \label{eq:Hineq}
	\sum_{m=1}^\infty \sum_{n=1}^\infty a_m \overline{a_n} \, K(m,n) \leq C \sum_{m=1}^\infty |a_m|^2
\end{equation}
holds for every square-summable sequence of complex numbers $a=(a_1,a_2,\ldots)$. The canonical example of an inequality of the form \eqref{eq:Hineq} is Hilbert's inequality, where the \emph{kernel} is $K(x,y)=(x+y)^{-1}$. In this case, it was Hilbert who proved that the inequality holds with the constant $C=2\pi$ before Schur \cite{Schur11} established that the best constant is $C=\pi$. It is for this reason that inequalities of the form \eqref{eq:Hineq} are commonly referred to as \emph{Hilbert-type inequalities}.

Schur actually established a much more general result. A hint to his approach can be found in our list of assumptions above, since the requirement that $K$ be a continuous function on $\mathbb{R}_+\times \mathbb{R_+}$ may seem incongruous as we only evaluate it at pairs of positive integers in \eqref{eq:Hineq}. Schur's main idea is to first study the continuous analogue of \eqref{eq:Hineq} and then reduce the continuous case to the discrete case.  

To state his result, we introduce the auxiliary function
\begin{equation} \label{eq:k}
	k(y) = \frac{K(1,y)}{\sqrt{y}}
\end{equation}
for $0<y<\infty$. The reduction from the continuous case to the discrete case goes through when $k$ enjoys certain monotonicity properties.

\begin{theorem}[Schur, 1911] \label{thm:schur}
	Let $K$ be a continuous, positive, and symmetric kernel of homogeneity $-1$ and let $k$ be as in \eqref{eq:k}. 
	\begin{enumerate}
		\item[(a)] If $k$ is decreasing on the interval $(1,\infty)$, then
		\[C(K) \geq \int_0^\infty k(y)\,dy.\]
		\item[(b)] If $k$ is decreasing on the interval $(0,\infty)$, then
		\[C(K) \leq \int_0^\infty k(y)\,dy.\]
	\end{enumerate}
\end{theorem}
The reader is invited to verify that $K(x,y)=(x+y)^{-1}$ satisfies the assumption of Theorem~\ref{thm:schur}~(b) and to check that the integral for $C(K)$ indeed equals $\pi$.

Chapter IX in the classical monograph \emph{Inequalities} by Hardy, Littlewood and P\'{o}lya \cite{HLP} is devoted to the further development of Schur's idea. Theorem~318 in that text is a generalization of (b), while (a) is implicitly contained in Section~IX.5. However, the authors are keenly aware of the limitations of Schur's approach:
  
\begin{quote}
	\textit{If the reader will try to deduce Theorem 331 from Theorem 328 similarly, he will find some difficulty. Something is lost in the passage from integrals to series, and it is by no means always possible that (as here) the passage can be made without damage to the final result.} \flushright Hardy, Littlewood and P\'{o}lya \cite[p.~249]{HLP}.
\end{quote}

There is a vast literature (which we make no attempt at delineating) of extensions and generalizations of Theorem~\ref{thm:schur} in various directions, common among them is that Schur's approach works with only superficial modifications. We will instead consider a family of kernels exemplifying the phenomenon quoted above, namely 
\[K_\alpha(x,y) = \frac{1}{\sqrt{xy}}\frac{(xy)^{\alpha}}{(\max(x,y))^{2\alpha}}\]
for $0<\alpha<\infty$. Accordingly, we let $C_\alpha=C(K_\alpha)$ denote the best constant in the inequality \eqref{eq:Hineq} with the kernel $K=K_\alpha$. Since
\[k_\alpha(y) = \begin{cases}
	y^{\alpha-1} & \text{if } 0<y\leq1, \\
	y^{-\alpha-1} & \text{if } 1< y < \infty,
\end{cases}\]
it is plain that $K_\alpha$ satisfies the assumption of Theorem~\ref{thm:schur}~(a) for all $0<\alpha <\infty$ and the assumption of Theorem~\ref{thm:schur}~(b) only for $0< \alpha \leq 1$. Consequently, the best constant satisfies $C_\alpha \geq 2/\alpha$ for all $0<\alpha<\infty$ and $C_\alpha = 2/\alpha$ when $0<\alpha\leq1$. We are interested in whether the latter conclusion holds for some $\alpha > 1$.

If $K=K_\alpha$, then the left-hand side of \eqref{eq:Hineq} enjoys the integral representation
\begin{equation} \label{eq:intrep}
	\alpha \sum_{m=1}^\infty \sum_{n=1}^\infty a_m \overline{a_n} \, K_\alpha(m,n) = \int_{-\infty}^\infty \left|\sum_{m=1}^\infty a_m m^{-\frac{1}{2}-it}\right|^2 \frac{\alpha^2}{\alpha^2+t^2}\,\frac{dt}{\pi},
\end{equation}
which can be established by expanding the absolute values into a double sum. Similar formulas can be obtained for other Hilbert-type inequalities via the Mellin transform (see e.g.~\cite[Sec.~IV]{Wilf64}). For each fixed square-summable sequence $a$, the right-hand side of \eqref{eq:intrep} is increasing as a function of $\alpha$. The same must also be true for the left-hand side, so the function $\alpha \mapsto \alpha C_\alpha$ is increasing. It follows that if $C_\alpha \leq 2/\alpha$ holds for some $\alpha$, then the same estimate also holds for all $0<\beta \leq \alpha$. 

To prove that $C_\alpha = 2/\alpha$ holds beyond $\alpha=1$, we will rely on another innovation of Schur's seminal paper \cite{Schur11}, namely the \emph{Schur test}. See \cite[Sec.~3]{DK03} for a historical account of the Schur test. In its simplest form, the Schur test is just the weighted Cauchy--Schwarz inequality with an unspecified weight. The strategy is to first use this inequality, then analyze the resulting expression and try to identify a good weight. Unfortunately, once a good weight is found the proof of the resulting inequality is often written up without any mention of how the weight was found. Indeed, Schur's (!) proof of Theorem~\ref{thm:schur} in Section~7 of \cite{Schur11} makes no reference to the Schur test first introduced in Section 3 of the very same paper.

The first goal of the present note is therefore to give a complete account of Theorem~\ref{thm:schur} including a clear explanation of how the weight is found. After analyzing how the proof of Theorem~\ref{thm:schur}~(b) fails for the kernels $K_\alpha$ when $\alpha>1$, we are next led by the Schur test to a sufficient condition for $C_\alpha \leq 2/\alpha$. We will finally use Euler--Maclaurin summation to show that this condition is satisfied for $\alpha=3/2$, thereby answering a question raised by the present author in \cite[Sec.~5.2]{Brevig17}.

\begin{theorem} \label{thm:32}
	For every square-summable sequence $a=(a_1,a_2,\ldots)$ it holds that
	\[\sum_{m=1}^\infty \sum_{n=1}^\infty a_m \overline{a_n}\, \frac{mn}{(\max(m,n))^3} \leq \frac{4}{3}\sum_{m=1}^\infty |a_m|^2\]
	and the constant $4/3$ cannot be replaced by any smaller number. 
\end{theorem}

By the fact that $\alpha \mapsto \alpha C_\alpha$ is increasing discussed above and Theorem~\ref{thm:schur}~(a), we also obtain the following.

\begin{corollary} \label{cor:32}
	Suppose that $0<\alpha\leq3/2$. For every square-summable sequence $a=(a_1,a_2,\ldots)$ it holds that
	\[\sum_{m=1}^\infty \sum_{n=1}^\infty a_m \overline{a_n}\, K_\alpha(m,n) \leq \frac{2}{\alpha}\sum_{m=1}^\infty |a_m|^2\]
	and the constant $2/\alpha$ cannot be replaced by any smaller number.
\end{corollary}

Corollary~\ref{cor:32} is an improvement on \cite[Thm.~1]{Brevig17} due to the present author, which implies that $C_\alpha = 2/\alpha$ for $0<\alpha \leq \alpha_0=1.48\ldots$. Here $\alpha_0$ denotes the unique positive solution of the equation $\alpha \zeta(1+\alpha)=2$ where $\zeta$ is the Riemann zeta function. The Riemann zeta function makes an appearance due to the relation
\begin{equation} \label{eq:zeta}
	\zeta(1+\alpha) = \sum_{n=1}^{\infty} n^{-1-\alpha} = \sum_{n=1}^\infty k_\alpha(n).
\end{equation}
Since $K_\alpha(1,1)=k_\alpha(1)=1$ for every $0<\alpha<\infty$, it is plain that $C_\alpha \geq 1$. Consequently, it is not true that $C_\alpha = 2/\alpha$ in general. 

Let us close out this introduction by briefly mentioning some interesting properties of the constants $C_\alpha$. The determination of $C_\alpha$ is equivalent to a problem arising in the theory of composition operators on the Hardy space of Dirichlet series through \cite[Prob.~3]{Hedenmalm04} and \cite[Thm.~2]{Brevig17}. The connection to composition operators provides at once the lower bound $C_\alpha \geq \zeta(1+2\alpha)$. The question of whether $C_\alpha = 2/\alpha$ is related to the discrete spectrum of certain Jacobi matrices by \cite[Thm.~C]{BPP23}, which in turn is related to the reproducing kernel thesis certain composition operators through material from \cite[Sec.~5]{MPQ18} and \cite[Sec.~5]{BPP23}. We refer to \cite[Ch.~8]{QQ20} for a general account of the theory of composition operators on Hardy spaces of Dirichlet series.

\section{The Schur test} 
We follow the strategy outlined by Schur \cite[p.~2]{Schur11} and begin by investigating the continuous analogue of \eqref{eq:Hineq}. Suppose that the kernel $K$ is continuous, positive, symmetric, and homogeneous of degree $-1$. Consider the inequality
\begin{equation} \label{eq:Hcont}
	\int_0^\infty \int_0^\infty f(x)\overline{f(y)}\, K(x,y)\,dydx \leq B \int_0^\infty |f(x)|^2\,dx.
\end{equation}
We want to find the smallest constant $B=B(K)$ such that the inequality \eqref{eq:Hcont} holds for every square-integrable complex-valued function $f$ on $\mathbb{R}_+$.

By the symmetry and positivity of the kernel $K$, we may assume without loss of generality that $f$ is nonnegative on $\mathbb{R}_+$. After inspecting \eqref{eq:Hcont}, it is natural to expect that the proof of such an estimate would involve the Cauchy--Schwarz inequality. A naive first attempt is to use the symmetry and positivity of $K$ to write
\[f(x) f(y) \, K(x,y) = f(x) \sqrt{K(x,y)} \, f(y) \sqrt{K(y,x)}\]
before applying the Cauchy--Schwarz inequality. It turns our that a slightly more refined approach is called for. A continuous function $\omega \colon \mathbb{R}_+ \to \mathbb{R}_+$ will be called a \emph{weight} in what follows. For an unspecified weight $\omega$, we write
\begin{equation} \label{eq:weight}
	f(x) f(y) \, K(x,y) = f(x) \sqrt{K(x,y)\frac{\omega(y)}{\omega(x)}} \, f(y) \sqrt{K(y,x)\frac{\omega(x)}{\omega(y)}}.
\end{equation}
By the Cauchy--Schwarz inequality and symmetry, we deduce from \eqref{eq:Hcont} and \eqref{eq:weight} that
\begin{equation} \label{eq:CS}
	\int_0^\infty \int_0^\infty f(x)\overline{f(y)}\, K(x,y)\,dydx \leq \int_0^\infty |f(x)|^2 \frac{1}{\omega(x)} \int_0^\infty K(x,y) \,\omega(y)\,dy dx.
\end{equation}
If we can find a weight $\omega$ and a constant $A$ such that the estimate
\begin{equation} \label{eq:schurcont}
	\int_0^\infty K(x,y) \, \omega(y)\,dy \leq A \omega(x)
\end{equation}
holds for every $0<x<\infty$, then plainly $B(K)\leq A$. This is the Schur test.

The plan is now to study the integral on the left-hand side of \eqref{eq:schurcont} in order to identify a suitable weight $\omega$. Due to the homogeneity of $K$, we can write 
\[\int_0^\infty K(x,y) \, \omega(y)\,dy = \int_0^\infty K(1,y) \, \omega(xy)\,dy.\]
From this we see that the easiest way to to attain the estimate \eqref{eq:schurcont} is to choose a weight which satisfies $\omega(xy)=\omega(x)\omega(y)$ for every $x$ and $y$ in $\mathbb{R}_+$. By the assumption that $\omega$ is continuous, this is only possible if $\omega(x)=x^r$ for a fixed real number $r$. 

We now want to pick $r$ to minimize the resulting integral. Appealing to the homogeneity of $K$ yet again, we find that
\[\int_0^\infty K(1,y)\, y^r \,dy = \int_1^\infty K(1,y) \left(y^r+y^{-r-1}\right)\,dy.\]
The minimum of the integrand on the right-hand side is attained at $r=-1/2$ for each fixed $1<y<\infty$. Hence it follows that
\begin{equation} \label{eq:Bint}
	B(K) \leq 2 \int_1^\infty \frac{K(1,y)}{\sqrt{y}}\,dy = \int_0^\infty \frac{K(1,y)}{\sqrt{y}}\,dy.
\end{equation}
Is this the best constant? The only estimate we have used is the Cauchy--Schwarz inequality in \eqref{eq:CS}. To attain the equality here with a non-trivial function, there must be a constant $C\neq0$ such that $f = C \omega$. However, this is not permissible since $\omega$ is not square-integrable on $\mathbb{R}_+$. To overcome this issue, we fix $\varepsilon>0$ and set
\[f_\varepsilon(x) = \begin{cases}
	0 & \text{if } 0<x\leq 1, \\
	x^{-\frac{1}{2}-\varepsilon} & \text{if } 1<x<\infty.
\end{cases}\]
In view of the final equality in \eqref{eq:Bint}, it is sufficient to consider a test function supported on $1<x<\infty$. It is plain that the square-integral of $f_\varepsilon$ on $\mathbb{R}_+$ is equal to $(2\varepsilon)^{-1}$. By the homogeneity of $K$ and integration by parts, we find that
\[\int_0^\infty \int_0^\infty f_\varepsilon(x)\overline{f_\varepsilon(y)}\,K(x,y)\,dydx = \frac{1}{\varepsilon} \int_1^\infty K(1,y)\,y^{-\frac{1}{2}-\varepsilon}\,dy.\]
Letting $\varepsilon\to 0^+$, we find that the estimate in \eqref{eq:Bint} indeed yields the best constant in \eqref{eq:Hcont}. We have consequently established the following result.

\begin{theorem} \label{thm:cont}
	Suppose that the kernel $K$ is continuous, positive, symmetric, and homogeneous of degree $-1$. For every square-integrable function $f$ on $\mathbb{R}_+$,
	\[\int_0^\infty \int_0^\infty f(x)\overline{f(y)} \, K(x,y)\,dydx \leq B \int_0^\infty |f(x)|^2\,dx, \qquad B = \int_0^\infty \frac{K(1,y)}{\sqrt{y}}\,dy,\]
	and the constant $B$ cannot be replaced by any smaller number.
\end{theorem}

Let us next turn to the discrete case and the proof of Theorem~\ref{thm:schur}. Although we will not explicitly use Theorem~\ref{thm:cont} in our proof, we are influenced by the choice of weight and test function made above.

\begin{proof}[Proof of Theorem~\ref{thm:schur} (a)]
	If $a=(a_m)_{m\geq1}$ is defined by $a_m = m^{-\frac{1}{2}-\varepsilon}$ for some $\varepsilon>0$, then $\sum_{m\geq1} |a_m|^2 = \zeta(1+2\varepsilon)$. Moreover,
	\[\sum_{m=1}^\infty \sum_{n=1}^\infty a_m \overline{a_n}\,K(m,n) = -K(1,1) \sum_{m=1}^\infty m^{-2-2\varepsilon} + 2 \sum_{m=1}^\infty \sum_{n=m}^\infty (mn)^{-\frac{1}{2}-\varepsilon} K(m,n).\]
	The first sum remains bounded as $\varepsilon\to0^+$ and can be ignored. We need to estimate the double sum from below, and we rewrite it using homogeneity to find that
	\[\sum_{m=1}^\infty \sum_{n=m}^\infty (mn)^{-\frac{1}{2}-\varepsilon} K(m,n) = \sum_{m=1}^\infty m^{-1-2\varepsilon} \sum_{n=m}^\infty \frac{1}{m} \frac{K(1,n/m)}{(n/m)^{\frac{1}{2}+\varepsilon}}.\]
	We recognize the inner sum as a \emph{left} Riemann sum with uniform partition size $m^{-1}$ for the integral 
	\[I_\varepsilon = \int_1^\infty \frac{K(1,y)}{y^{\frac{1}{2}+\varepsilon}}\,dy = \int_1^\infty \frac{k(y)}{y^\varepsilon}\,dy.\]
	Combining the assumption that $k$ is decreasing on the interval $(1,\infty)$ with the fact that the function $y \mapsto y^{-\varepsilon}$ is decreasing on the same interval and a geometric argument, it can be seen that $I_\varepsilon$ is a lower bound for every left Riemann sum (see Figure~\ref{fig:rsum12} for an example). Thus,
	\[\sum_{m=1}^\infty \sum_{n=1}^\infty a_m \overline{a_n}\,K(m,n) \geq -K(1,1) \zeta(2+2\varepsilon) + 2I_\varepsilon \zeta(1+2\varepsilon).\]
	Letting $\varepsilon\to0^+$, we find that $C(K)\geq 2I_0$. We finish the proof by using the homogeneity of $K$ as in \eqref{eq:Bint}.
\end{proof}

Let us now turn to the Schur test in the discrete case. As above, if there is a \emph{weight} $\omega\colon \mathbb{N}\to \mathbb{R}_+$ and a constant $A$ such that
\begin{equation} \label{eq:schurdiscrete}
	\sum_{n=1}^\infty K(m,n)\, \omega(n) \leq A \omega(m)
\end{equation}
holds for every $m\geq1$, then the best constant $C$ in the Hilbert-type inequality \eqref{eq:Hineq} satisfies $C \leq A$.

\begin{proof}[Proof of Theorem~\ref{thm:schur} (b)]
	Following our analysis of the continuous case above it is natural to choose the weight $\omega(m) = 1/\sqrt{m}$ for $m\geq1$. By \eqref{eq:schurdiscrete}, we then find that 
	\begin{equation} \label{eq:supme}
		C(K) \leq \sup_{m\geq1}\frac{1}{\omega(m)}\sum_{n=1}^\infty K(m,n)\, \omega(n) = \sup_{m\geq1}\sum_{n=1}^\infty \frac{1}{m} \frac{K(1,n/m)}{\sqrt{n/m}}.
	\end{equation}
	We recognize the right-hand side of \eqref{eq:supme} as \emph{right} Riemann sums of uniform partition size $m^{-1}$ for the integral
	\[I = \int_0^\infty \frac{K(1,y)}{\sqrt{y}}\,dy = \int_0^\infty k(y)\,dy.\]
	The assumption that $k$ is decreasing on the interval $(0,\infty)$ and a geometric argument (see Figure~\ref{fig:rsum1} for an example) shows that the supremum in \eqref{eq:supme} is attained as $m\to\infty$ and is equal to $I$. Hence we get that $C(K) \leq I$ by the Schur test.
\end{proof}

\begin{figure}
	\centering
	\begin{tikzpicture}[scale=1.5]
		\begin{axis}[
			axis equal image,
			axis lines = none,
			xmin = 1/2, 
			xmax = 11/2,
			ymin = -1, 
			ymax = 2]
			
			\draw[solid,draw opacity=0,fill=black,fill opacity=0.25] (1,0) rectangle (3/2,{1^(-3/2-1/4)});
			\draw[solid,draw opacity=0,fill=black,fill opacity=0.25] (3/2,0) rectangle (2,{(3/2)^(-3/2-1/4)});
			\draw[solid,draw opacity=0,fill=black,fill opacity=0.25] (2,0) rectangle (5/2,{2^(-3/2-1/4)});
			\draw[solid,draw opacity=0,fill=black,fill opacity=0.25] (5/2,0) rectangle (3,{(5/2)^(-3/2-1/4)});
			\draw[solid,draw opacity=0,fill=black,fill opacity=0.25] (3,0) rectangle (7/2,{3^(-3/2-1/4)});
			\draw[solid,draw opacity=0,fill=black,fill opacity=0.25] (7/2,0) rectangle (4,{(7/2)^(-3/2-1/4)});
			\draw[solid,draw opacity=0,fill=black,fill opacity=0.25] (4,0) rectangle (9/2,{4^(-3/2-1/4)});
			\draw[solid,draw opacity=0,fill=black,fill opacity=0.25] (9/2,0) rectangle (5,{(9/2)^(-3/2-1/4)});

			\addplot[solid] coordinates {(1/2,0) (11/2,0)};
			\addplot[solid,domain=1:5,samples=1000] {x^(-3/2-1/4)};
		\end{axis}
	\end{tikzpicture}
	\caption{Left Riemann sums of uniform partition size $m^{-1}$ of the function $y \mapsto y^{-\varepsilon} k_\alpha(y)$ on the interval $(1,5)$. Here $m=2$, $\alpha=1/2$, and $\varepsilon=1/4$.}
	\label{fig:rsum12}
\end{figure}
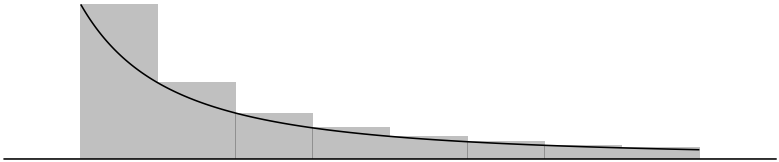

A similar Riemann sum argument starting from \eqref{eq:supme} under the assumption that $k$ is increasing on $(0,1)$ and decreasing on $(1,\infty)$, gives that
\begin{equation} \label{eq:badest}
	C(K) \leq k(1) + \int_1^\infty k(y)\,dy.
\end{equation}
This estimate is in general unlikely to be sharp, since we have to choose $m=1$ to attain the supremum for the integral over the interval $(0,1)$ and $m\to\infty$ to attain the supremum for the integral over the interval $(1,\infty)$. For the kernels $K_\alpha$ the estimate \eqref{eq:badest} becomes $C_\alpha \leq 1+1/\alpha$. This is sharp if and only if $\alpha=1$ when $k_1(y)=1$ for all $0<y<1$. This case is presented in Figure~\ref{fig:rsum1}. 

The typical situation for $\alpha>1$ is presented in Figure~\ref{fig:rsum32}. One possible strategy to improve \eqref{eq:badest} is to keep track of the overestimates on $(0,1)$ and underestimates on $(1,\infty)$ for each fixed $m\geq1$ and try to compute the supremum in \eqref{eq:supme}. This plan was carried out in \cite[Lem.~8]{Brevig17} and it led to the proof that $C_\alpha = 2/\alpha$ for $0 < \alpha \leq 1.48\ldots$ mentioned in the introduction. We will instead take an alternative approach, which in addition to giving a stronger result also is somewhat easier to handle from a computational point of view. 

Our approach is based on the observation is that $K_\alpha(1,1)=k_\alpha(1)=1$, so as $\alpha$ increases this term will be increasingly dominant. The plan is therefore to adjust the weight in the Schur test accordingly. The key idea of the next result is that we \emph{require} of the unspecified weight $\omega$ that \eqref{eq:schurdiscrete} is satisfied with $A = 2/\alpha$. We then try to choose a parameter in the weight in order to satisfy this requirement.

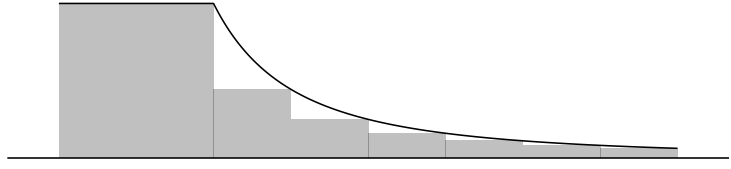
\begin{figure}
	\centering
	\begin{tikzpicture}[scale=1.5]
		\begin{axis}[
			axis equal image,
			axis lines = none,
			xmin = -1/2, 
			xmax = 9/2,
			ymin = -1, 
			ymax = 2]
			
			\draw[solid,draw opacity=0,fill=black,fill opacity=0.25] (0,0) rectangle (1,1);
			
			\draw[solid,draw opacity=0,fill=black,fill opacity=0.25] (1,0) rectangle (3/2,{(3/2)^(-2)});
			\draw[solid,draw opacity=0,fill=black,fill opacity=0.25] (3/2,0) rectangle (2,{2^(-2)});
			\draw[solid,draw opacity=0,fill=black,fill opacity=0.25] (2,0) rectangle (5/2,{(5/2)^(-2)});
			\draw[solid,draw opacity=0,fill=black,fill opacity=0.25] (5/2,0) rectangle (3,{3^(-2)});
			\draw[solid,draw opacity=0,fill=black,fill opacity=0.25] (3,0) rectangle (7/2,{(7/2)^(-2)});
			\draw[solid,draw opacity=0,fill=black,fill opacity=0.25] (7/2,0) rectangle (4,{4^(-2)});
			
			\addplot[solid] coordinates {(-1/3,0) (13/3,0)};
			\addplot[solid,domain=0:4,samples=1000] {1/(max(1,x))^2};
		\end{axis}
	\end{tikzpicture}
	\caption{Right Riemann sums of uniform partition size $m^{-1}$ of the function $k_\alpha$ on the interval $(0,4)$. Here $m=2$ and $\alpha=1$.}
	\label{fig:rsum1}
\end{figure}

\begin{lemma} \label{lem:schurKa}
	Suppose that $0<\alpha<2$. If
	\begin{equation} \label{eq:schurest}
		\zeta(1+\alpha) \leq \frac{2}{\alpha} + \left(\frac{2}{\alpha}-1\right) m^\alpha \left(\frac{2}{\alpha}  - \sqrt{m}\sum_{n=1}^\infty \frac{K_\alpha(m,n)}{\sqrt{n}}\right)
	\end{equation}
	holds for every $m\geq2$, then $C_\alpha = 2/\alpha$.
\end{lemma}

\begin{proof}
	We already know that $C_\alpha \geq 2/\alpha$ by Theorem~\ref{thm:schur} (a), so it is enough to consider the upper bound $C_\alpha \leq 2/\alpha$. We will use the Schur test \eqref{eq:schurdiscrete} with the weight
	\[\omega_\alpha(m) = \begin{cases}
		\delta_\alpha & \text{if } m=1, \\
		1/\sqrt{m} & \text{if } m\geq2,
	\end{cases}\]
	for some parameter $\delta_\alpha>0$. To conclude that $C_\alpha\leq 2/\alpha$, we need to establish that
	\begin{equation} \label{eq:schuralpha}
		\sum_{n=2}^\infty K_\alpha(m,n) \, \omega_\alpha(n) \leq \frac{2}{\alpha} \omega_\alpha(m)
	\end{equation}
	holds for every $m\geq1$. There are two cases. First, if $m=1$, then \eqref{eq:schuralpha} yields the requirement
	\[\delta_\alpha + \sum_{n=2}^\infty \frac{K_\alpha(1,n)}{\sqrt{n}} \leq \frac{2}{\alpha} \delta_\alpha.\]
	Note that for $0<\alpha<2$ we may choose a positive $\delta_\alpha$ satisfying this estimate as required by the Schur test. Second, if $m\geq2$, then \eqref{eq:schuralpha} yields the requirement
	\[\frac{\delta_\alpha}{m^{\alpha+\frac{1}{2}}} + \sum_{n=2}^\infty \frac{K_\alpha(m,n)}{\sqrt{n}}\leq \frac{2}{\alpha}\frac{1}{\sqrt{m}}.\]
	We can find $\delta_\alpha>0$ which satisfies both requirements whenever
	\[\left(\frac{2}{\alpha}-1\right)^{-1}\sum_{n=2}^\infty \frac{K_\alpha(1,n)}{\sqrt{n}} \leq \frac{2}{\alpha} m^\alpha - m^{\alpha+\frac{1}{2}}\sum_{n=2}^\infty \frac{K_\alpha(m,n)}{\sqrt{n}}\]
	for every $m\geq2$. This is equivalent to \eqref{eq:schurest} by a computation involving \eqref{eq:zeta}.
\end{proof}

\begin{figure}
	\centering
	\begin{tikzpicture}[scale=1.5]
		\begin{axis}[
			axis equal image,
			axis lines = none,
			xmin = -1/3, 
			xmax = 13/3,
			ymin = -1, 
			ymax = 2]
			
			\draw[solid,draw opacity=0,fill=black,fill opacity=0.25] (0,0) rectangle (1/3,{sqrt(1/3)});
			\draw[solid,draw opacity=0,fill=black,fill opacity=0.25] (1/3,0) rectangle (2/3,{sqrt(2/3)});
			\draw[solid,draw opacity=0,fill=black,fill opacity=0.25] (2/3,0) rectangle (1,1);
			
			\draw[solid,draw opacity=0,fill=black,fill opacity=0.25] (1,0) rectangle (4/3,{(4/3)^(-5/2)});
			\draw[solid,draw opacity=0,fill=black,fill opacity=0.25] (4/3,0) rectangle (5/3,{(5/3)^(-5/2)});
			\draw[solid,draw opacity=0,fill=black,fill opacity=0.25] (5/3,0) rectangle (2,{2^(-5/2)});
			
			\draw[solid,draw opacity=0,fill=black,fill opacity=0.25] (2,0) rectangle (7/3,{(7/3)^(-5/2)});
			\draw[solid,draw opacity=0,fill=black,fill opacity=0.25] (7/3,0) rectangle (8/3,{(8/3)^(-5/2)});
			\draw[solid,draw opacity=0,fill=black,fill opacity=0.25] (8/3,0) rectangle (3,{3^(-5/2)});
			
			\draw[solid,draw opacity=0,fill=black,fill opacity=0.25] (3,0) rectangle (10/3,{(10/3)^(-5/2)});
			\draw[solid,draw opacity=0,fill=black,fill opacity=0.25] (10/3,0) rectangle (11/3,{(11/3)^(-5/2)});
			\draw[solid,draw opacity=0,fill=black,fill opacity=0.25] (11/3,0) rectangle (4,{4^(-5/2)});
			
			\addplot[solid] coordinates {(-1/3,0) (13/3,0)};
			\addplot[solid,domain=0:4,samples=1000] {sqrt(x)/(max(1,x))^3};
		\end{axis}
	\end{tikzpicture}
	\caption{Right Riemann sums of uniform partition size $m^{-1}$ of the function $k_\alpha$ on the interval $(0,4)$. Here $m=3$ and $\alpha=3/2$.}
	\label{fig:rsum32}
\end{figure}
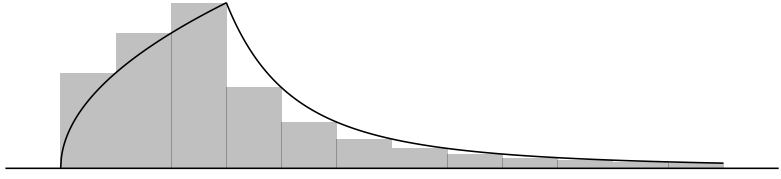

\section{Euler--Maclaurin summation}
We require two estimates in order to establish that the requirement \eqref{eq:schurest} from Lemma~\ref{lem:schurKa} holds for $\alpha=3/2$. These estimates be extracted from \cite[Sec.~4]{Brevig17}, but we include a complete account here to ensure that the present note is self-contained. Although we shall use the two estimates only for $\alpha=3/2$ in the proof of Theorem~\ref{thm:32}, we state them somewhat generally. What we need to know about Euler--Maclaurin summation can be found in \cite[Sec.~11.5]{Duren12}.

\begin{lemma} \label{lem:EM1}
	Fix $0<\alpha<\infty$. For every $m\geq1$ it holds that 
	\[\sum_{n=m+1}^\infty n^{-\alpha-1} \leq \frac{1}{\alpha} m^{-\alpha}-\frac{1}{2} m^{-\alpha-1}+\frac{(\alpha+1)}{12} m^{-\alpha-2}.\]
\end{lemma}

\begin{proof}
	Let $f$ be a function defined on the interval $[m,\infty)$ which has continuous derivatives of order three on the same interval. If both $f$ and $f'$ vanish at infinity, then one step of the Euler--Maclaurin summation formula yields that
	\begin{equation} \label{eq:EM11}
		\sum_{n=m+1}^\infty f(n) = \int_m^\infty f(x)\,dx - \frac{f(m)}{2} - \frac{f'(m)}{12} + \frac{1}{3!} \int_m^\infty b_3(\{x\}) f^{(3)}(x)\,dx,
	\end{equation}
	where $\{x\}$ denotes the fractional part of $x$ and
	\[b_3(x) = x^3 - \frac{3}{2}x^2+\frac{1}{2}x\]
	is the third Bernoulli polynomial. If $0 \leq x \leq 1/2$, then $b_3(1-x) = - b_3(x)$ and $b_3(x) \geq 0$. If $f^{(3)}$ is an increasing function on the interval $[m,\infty)$, then a symmetry consideration shows that
	\begin{equation} \label{eq:EM12}
		\int_m^\infty b_3(\{x\}) f^{(3)}(x)\,dx  \leq 0.
	\end{equation}
	We apply \eqref{eq:EM11} and \eqref{eq:EM12} to $f(x) = x^{-\alpha-1}$ and obtain the stated result.
\end{proof}

\begin{lemma} \label{lem:EM2}
	Fix $1 \leq \alpha \leq 2$. For every $m\geq1$ it holds that
	\[\sum_{n=1}^m n^{\alpha-1} \leq \frac{1}{\alpha}m^{\alpha} + \frac{1}{2} m^{\alpha-1} + \frac{\alpha-1}{12} m^{\alpha-2} - \frac{(3-\alpha)(5-\alpha)(6-\alpha)(8+\alpha)}{720\alpha}.\]
\end{lemma}

\begin{proof}
	Let $f$ be a function defined on the interval $[1,m]$ which has continuous derivatives of order five on the same interval. Two steps of the Euler--Maclaurin summation formula yields that
	\begin{equation} \label{eq:EM21}
		\begin{split}
			\sum_{n=1}^m f(n) = \int_1^m f(x)\,dx &+ \frac{f(1)+f(m)}{2} + \frac{f'(m)-f'(1)}{12}  \\ 
			&- \frac{f^{(3)}(m)-f^{(3)}(1)}{720} + \frac{1}{5!} \int_1^m b_5(\{x\}) f^{(5)}(x)\,dx,
		\end{split}
	\end{equation}
	where $\{x\}$ again denotes the fractional part of $x$ and 
	\[b_5(x) = x^5 - \frac{5}{2}x^4+\frac{5}{3}x^3-\frac{1}{6}x\]
	is the fifth Bernoulli polynomials. If $0 \leq x \leq 1/2$, then $b_5(1-x) = - b_5(x)$ and $b_5(x) \leq 0$. If $f^{(5)}$ is a decreasing function on $[1,m]$, then a symmetry consideration shows that
	\begin{equation} \label{eq:EM22}
		\int_1^m b_5(\{x\}) f^{(5)}(x)\,dx  \leq 0.
	\end{equation}
	If $1 \leq \alpha \leq 2$ and $f(x) = x^{\alpha-1}$, then $f^{(5)}$ is decreasing on $[1,m]$ for every $m\geq2$. We therefore obtain the stated result from \eqref{eq:EM21}, \eqref{eq:EM22}, and the estimate 
	\[-\frac{f^{(3)}(m)}{720} = - \frac{(\alpha-1)(\alpha-2)(\alpha-3)}{720} m^{\alpha-4} \leq 0,\]
	which holds when $1 \leq \alpha \leq 2$.
\end{proof}

We can now deduce our main result from Lemma~\ref{lem:schurKa}, Lemma~\ref{lem:EM1}, and Lemma~\ref{lem:EM2}.

\begin{proof}[Proof of Theorem~\ref{thm:32}]
	By Lemma~\ref{lem:schurKa}, it is sufficient to establish the estimate \eqref{eq:schurest} for $\alpha=3/2$ and every $m\geq2$. We begin by estimating the left-hand side of \eqref{eq:schurest} from above using Lemma~\ref{lem:EM1} with $\alpha=3/2$ and $m=4$ to obtain
	\[\zeta({\textstyle \frac{5}{2}}) \leq 1 + 2^{-5/2}+3^{-5/2}+4^{-5/2} + 4^{-3/2}\left(\frac{2}{3}-\frac{1}{8}+\frac{5}{384}\right)=\frac{\sqrt{2}}{8}+\frac{\sqrt{3}}{27}+\frac{1127}{1024}.\]
	To estimate the right-hand side of \eqref{eq:schurest}, we first compute
	\begin{align*}
		\sqrt{m}\sum_{n=1}^\infty \frac{K_\alpha(m,n)}{\sqrt{n}} &= m^{-\alpha} \sum_{n=1}^m n^{\alpha-1} + m^\alpha \sum_{n=m+1}^\infty n^{-\alpha-1}, \\
		\intertext{then use Lemma~\ref{lem:EM1}, Lemma~\ref{lem:EM2} for $1 \leq \alpha \leq 2$, and finally that $m\geq2$ to obtain}
		&\leq \frac{2}{\alpha} + \frac{\alpha}{6}\frac{1}{m^2} - \frac{(3-\alpha)(5-\alpha)(6-\alpha)(8+\alpha)}{720\alpha}\frac{1}{m^\alpha} \\
		&\leq \frac{2}{\alpha} + \left(\frac{\alpha}{6} 2^{\alpha-2} - \frac{(3-\alpha)(5-\alpha)(6-\alpha)(8+\alpha)}{720\alpha}\right)\frac{1}{m^\alpha}.
	\end{align*}
	Inserting this estimate into the right-hand side of \eqref{eq:schurest} and setting $\alpha=3/2$, we get
	\[\frac{2}{3/2} + \left(\frac{2}{3/2}-1\right) m^{3/2} \left(\frac{2}{3/2}  - \sqrt{m}\sum_{n=1}^\infty \frac{K_{3/2}(m,n)}{\sqrt{n}}\right) \geq \frac{2693}{1920}-\frac{\sqrt{2}}{24}.\]
	This completes the proof since
	\[\frac{\sqrt{2}}{8}+\frac{\sqrt{3}}{27}+\frac{1127}{1024} < \frac{2693}{1920}-\frac{\sqrt{2}}{24}. \qedhere\]
\end{proof}

What does this mean for the parameter $\delta_\alpha$ in the proof of Lemma~\ref{lem:schurKa}? Inserting the estimates from Lemma~\ref{lem:EM1} and Lemma~\ref{lem:EM2} directly into the requirements, we find that an acceptable choice satisfies
\[1.0245\ldots = 3 \left(\frac{\sqrt{2}}{8}+\frac{\sqrt{3}}{27}+\frac{103}{1024}\right) \leq \delta_{3/2} \leq \frac{773}{640}-\frac{\sqrt{2}}{8} = 1.0315\ldots\]
It is not possible to push much further with Lemma~\ref{lem:schurKa}: a numerical computation shows that \eqref{eq:schurest} does not hold for $m=2$ when $\alpha \geq 1.5069$.

\bibliographystyle{amsplain} 
\bibliography{hilbertineq}

%

\end{document}